\documentclass[12pt,reqno]{amsart}
\usepackage[T1]{fontenc}
\usepackage[latin9]{inputenc}
\usepackage{multirow}
\usepackage{array}
\usepackage{amsmath, amssymb}
\usepackage{amsthm} 

\theoremstyle{plain} 
\newtheorem{theorem}{Theorem}[section]

\newtheorem{corollary}[theorem]{Corollary}
\newtheorem{proposition}[theorem]{Proposition}

\theoremstyle{definition}
\newtheorem{definition}[theorem]{Definition}
\newtheorem{remark}[theorem]{Remark}

 \usepackage{caption}

\usepackage{color}
\usepackage{graphicx,color}
\usepackage{amsmath, amssymb, graphics}

\def\r{\mathbb R}
\def\z{\mathbb Z}
\def\n{\mathbb N}

\begin{document}
\title[The number of catenoids connecting two coaxial circles]{The number of catenoids connecting two coaxial circles in Lorentz-Minkowski space}
  \author[S. Akamine]{Shintaro Akamine}
   \address{Graduate School of Mathematics Kyushu University\\
744 Motooka, Nishi-ku,\
Fukuoka 819-0395, Japan}
 \email{s-akamine@math.kyushu-u.ac.jp}
 \author{Rafael L\'opez}
  \address{Departamento de Geometr\'{\i}a y Topolog\'{\i}a\\ Instituto de Matem\'aticas (IEMath-GR)\\
 Universidad de Granada\\
 18071 Granada, Spain}
 \email{rcamino@ugr.es}


\begin{abstract}In $3$-dimensional Lorentz-Minkowski space we determine the number of catenoids connecting two coaxial circles in parallel planes. This study is separated according to the types of circles and the causal character (spacelike and timelike) of the catenoid.
\end{abstract}

\maketitle

\section{Introduction and statement of results}

The catenoid is the only non-planar  rotational minimal surface in Euclidean space and it is generated by the catenary $f(z)=(1/a)\cosh(az+b)$ when rotates around the  $z$ axis. Consider a piece of a catenoid   bounded by   two coaxial circles $C_1\cup C_2$ with the same radius $r>0$ and separated a distance $h>0$ far apart. It is known that if we go separating $C_1$ from $C_2$, there is a critical  distance between $C_1$ and $C_2$ where the catenoid breaks into two circular disks around each circle $C_i$. The relation between $r$ and $h$ is the following: there exists a value $c_1\simeq 1.325$ such that if $h/r<c_1$, there are exactly two catenoids connecting $C_1$ and $C_2$, if $h/r=c_1$, there is exactly one and if $h/r>c_1$, there is no a catenoid spanning $C_1\cup C_2$ (see for example \cite{bl,is,ni}).  Related with the above phenomenon, there is the question to determine if a catenoid is a minimizer of surface area because in general, one of the two catenoids is not  a absolute minimizer.  Exactly, there exists a value $c_2\simeq 1.056$ such that if   if $h/r<c_2$, then  there exists a unique catenoid spanning $C_1\cup C_2$ that is an absolute minimum for the surface area but if $h/r> c_2$, then the two disks spanning $C_i$ give an absolute minimum for surface area (the so-called  Goldschmidt discontinuous solution).
Notice that if $c_2<h/r<c_1$, then the catenoid is only a local minimum.

 In this paper we consider in $3$-dimensional Lorentz-Minkowski space $\r_1^3$ the problem on the number of catenoids connecting two coaxial circles. In this setting, we need to precise the  above notions. First,  it is the definition of a rotational surface. In $\r_1^3$ there are three types of uniparametric groups of rotations depending on the causal character of the rotation axis and are called \emph{elliptic}, \emph{hyperbolic} and \emph{parabolic} when the rotation axis is timelike, spacelike and  lightlike respectively.  In particular, in $\r_1^3$ there are three types of  rotational surfaces. We call a \emph{circle} of $\r_1^3$ the orbit of a point under some of the above groups of rotations when such as an orbit is not a straight line. On the other hand, the notion of the  mean curvature is defined in a surface whose induced metric from $\r_1^3$ is not degenerate, that is, when the surface is spacelike (Riemannian metric) and the surface is timelike (Lorentzian metric).   A \emph{catenoid} is a non-degenerate  rotational surface  with zero mean curvature everywhere.  
 
 The problem that we study is the following:
 \begin{quote} Given two coaxial circles in Lorentz-Minkowski space, how many catenoids span both circles?
 \end{quote}

By two coaxial circles we mean two circles placed in different planes and that are invariant by the same group of rotations.

 For understanding our main result (Theorem \ref{t-main}) we need to point out some observations. If two circles are coaxial, then they  are invariant by one of the three groups of rotations, but not for the other two ones (see Sect. \ref{sec2} for the description of the circles in $\r_1^3$).  On the other hand, the causal character  of the circles imposes restrictions on the (possible) catenoid that span. For example, if the two circles are timelike curves, then the catenoid can not be spacelike. 
 
 We will prove in some cases that the number of catenoids connecting the circles is $0$ or $1$. In this situation we will assume  coaxial circles with arbitrary radius. However, in other cases  there exist many catenoids connecting two coaxial circles and this number increases as the separation distance increases (timelike elliptic catenoids and spacelike hyperbolic catenoids of type II; see Sect. \ref{sec2} below). Then we suppose here that the  circles have the same radius.  

In the following sections, we will state in a precise manner the results obtained according to the  type of the rotation group  and  the causal character  of the surface (spacelike or timelike): see Theorems \ref{t-elliptic}, \ref{t-elliptic2}, \ref{t-hyperbolic}, \ref{t-hyperbolic2} and \ref{t-parabolic}. We can now give a general view of the results in the next theorem and the corresponding Table \ref{table1}.

\begin{theorem}\label{t-main}
 Let $C_1$ and $C_2$ be two coaxial circles in Lorentz-Minkowski space $\r_1^3$.
\begin{enumerate}
\item There exists $0$ or $1$ catenoid connecting $C_1$ and $C_2$ in the following cases: spacelike elliptic catenoid, timelike hyperbolic catenoid of type II and  parabolic catenoid.
\item There exist $0$, $1$ or $2$ timelike hyperbolic catenoids of type I.
\item For timelike elliptic catenoids and spacelike hyperbolic catenoids of type II, and if the circles have the same radius, there exists a number $N(h)\geq 1$ of catenoids connecting $C_1$ and $C_2$ depending on the distance $h$ between the circles, where $N(h)$ is non decreasing on $h$ and $\lim_{h\rightarrow \infty}N(h)=\infty$.
    \end{enumerate}
\end{theorem}

\begin{table}[htbp]
\centering \renewcommand{\arraystretch}{2}
\begin{tabular}{|l|c|c|c|c|}
\hline & \multicolumn{4}{c|}{Types of rotational surfaces} \\
\cline{2-5}
& &\multicolumn{2}{|c|}{hyperbolic}&\\
\cline{3-4}
     & elliptic & type I  & type II& parabolic \\\hline
    spacelike  & $0,1$	&	  -- &	 $N(h)^*$ & $0,1$\\       \hline
  timelike & $N(h)^*$&	 $0,1,2$ &	 $0,1$ & $0,1$\\
    \hline
\end{tabular}
\caption{The number of catenoids connecting two coaxial circles according to the type of rotation group and the type of causality of the surface. In (*), the radius of the circles coincide}
\label{table1}
\end{table}

 \section{Catenoids in Lorentz-Minkowski space}\label{sec2}

Consider the Lorentz-Minkowski space $\r_1^3=(\r^3,\langle,\rangle=dx^2+dy^2-dz^2)$
where $(x, y, z)$ are the canonical coordinates in $\mathbb{R}^3$. A vector $v\in\r_1^3$ is said to be spacelike (resp. timelike, lightlike) if   $\langle v, v \rangle > 0$ or $v=0$ (resp. $\langle v, v \rangle < 0$,  $\langle v, v \rangle = 0$ and $v\neq 0$). In $\r_1^3$ there are three types of uniparametric groups of isometries that leave pointwise fixed a straight line $L$. In order to give a description of such groups, we do a change of coordinates and we suppose that $L$ is given in terms of the canonical basis  of $\mathbb{R}^3$, namely,  $B=\{e_1, e_2, e_3\}=\{(1,0,0),(0,1,0),(0,0,1)\}$. Let  $\{A(t):t\in\r\}$ be the uniparametric group of isometries whose rotation axis is $L$, where $A(t)$ denotes the isometry as well as the matricial expression with respect to $B$. Then we have the next classification according   the causal character of $L$:

\begin{enumerate}
\item The axis is timelike, $L=\mbox{sp}\{e_3\}$. Then
$$ A(t)=\left(\begin{array}{ccc}
\cos t&-\sin t&0\\ \sin t&\cos t&0\\
0& 0&1\end{array}\right).$$

\item The axis is spacelike, $L=\mbox{sp}\{e_1\}$. Then
$$ A(t)=\left(\begin{array}{ccc}1&0&0\\
0&\cosh t&\sinh t\\
0&\sinh t&\cosh t\end{array}\right).$$
\item The axis is lightlike, $L=\mbox{sp}\{e_1+e_3\}$. Then
$$ A(t)=\left(\begin{array}{lll}1-\frac{t^2}{2}&t&\frac{t^2}{2}\\ -t&1&t\\-\frac{t^2}{2}&t&1+\frac{t^2}{2}\end{array}\right).$$
\end{enumerate}

A \emph{circle} in $\r_1^3$ is the orbit of a point under one of the above groups  when the orbit is not a straight line. In particular, this implies that  the point does  belong to the rotation axis. With respect to the above choices of rotation axes $L$, a circle describes an  Euclidean circle (resp. hyperbola,  parabola) if $L$ is timelike (resp. spacelike, lightlike).  Exactly, and depending on the  causal character  of $L$, we have:
\begin{enumerate}
\item Timelike axis.  Each circle meets the $xz$-plane. Let $(a,0,c)$ be a point in this plane that does not belong to $L$  ($a\not=0$). The orbit is the circle $\alpha(t)=(0,0,c)+r(\cos(t),\sin(t),0)$ where $r=|a|>0$  is called the {\it radius} of $\alpha$.
    \item Spacelike axis.  Each circle meets the $xy$-plane  or the $xz$-plane. Let  $(a,c,0)$ (resp. $(a,0,c)$) be a point that does not belong to $L$, that is, $c\not=0$.  The orbit is the circle $\alpha(t)=(a,0,0)+r(0,\cosh(t),\sinh(t))$ (resp.  $\alpha(t)=(a,0,0)+r(0,\sinh(t),\cosh(t))$) where $r=|c|>0$ is called the {\it radius} of $\alpha$.
        \item Lightlike axis.  Each circle meets the $xz$-plane. Consider a point $(a,0,c)$ that does not belong to the axis ($a-c\not=0$). Then the circle is a parabola in a parallel plane to the plane of equation $x-z=0$ and parametrized by $\alpha(t)= (a,0,c)+t(0,1,0)+ \frac{t^2}{2(c-a)}(1,0,1)$. Here we do not define the center and the radius of the circle. The circle $\alpha$ is a spacelike curve.

        \end{enumerate}

Once obtained the three groups of rotations of $\r_1^3$, we give the description of a local parametrization $X(s,t)$ of a rotational surface. Using  the terminology given in the introduction, we obtain the next classification:
\begin{proposition} \label{p-revolution}
Up to an isometry of $\r_1^3$, a local parametrization of a  rotational surface is given as follows: if $f\in C^{\infty}(I)$, $I\subset\r$ and $s,t\in\r$, then:
\begin{enumerate}
\item Elliptic rotational surface. The parametrization is
$$X(s,t)=A(t)\cdot(f(s),0,s)=(f(s)\cos t,f(s)\sin t,s).$$
 The circles are Euclidean circles   contained in parallel   planes to the $xy$-plane.
\item Hyperbolic rotational surface. We have two subcases:
\begin{enumerate}
\item Type I. The parametrization is
$$X(s,t)=A(t)\cdot(s,f(s),0)=(s,f(s)\cosh t,f(s) \sinh t)$$
 and the circles are timelike hyperbolas contained in parallel planes to the $yz$-plane.
    \item Type II.  The parametrization is
    $$X(s,t)=A(t)\cdot(s,0,f(s))=(s,f(s)\sinh t,f(s)\cosh t)$$
     and the circles are spacelike hyperbolas contained in parallel  planes to the $yz$-plane.

\end{enumerate}
\item Parabolic rotational surface. The parametrization is
$$X(s,t)=A(t)\cdot(f(s)+s,0,f(s)-s)$$ and the circles are parabolas contained in planes parallel to the plane of equation $x-z=0$. The curve of  vertices of these parabolas lies included in the plane of equation $y=0$ and it is a graph on the line $\mbox{sp}\{(1,0,1)\}$, namely, $s\longmapsto s(1,0,-1)+f(s)(1,0,1)$.
\end{enumerate}
\end{proposition}

We now recall the notion of the mean curvature for a non-degenerate surface in $\r_1^3$.  See the references \cite{on} and \cite{we} for details.  If $X: M \rightarrow \r_1^3$ is an immersion of a  smooth surface $M$, we say that $X$ is non-degenerate if the induced metric on $M$ is non-degenerated. There are only two possibilities of non-degenerate surfaces: the metric is Riemannian and  we say that $X$ is \emph{spacelike}, or   the metric is Lorentzian and we say that $X$ is \emph{timelike}. In terms of the coefficients of the first fundamental form of $X$, namely,  $E=\langle X_s,X_s\rangle$, $F=\langle X_s,X_t\rangle$ and $G=\langle X_t,X_t\rangle$, the surface is spacelike if $EG-F^2>0$ and it is timelike if  $EG-F^2<0$. The mean curvature $H$ is defined as the trace of the second fundamental form. If  $X=X(s,t)$ is  a local parametrization,  the zero mean curvature equation $H=0$  writes as
$$E\ \mbox{det}(X_s,X_t,X_{tt})-2F\ \mbox{det}(X_s,X_t,X_{st})+G\ \mbox{det}(X_s,X_t,X_{ss})=0.$$

\begin{definition}
A  \emph{catenoid} in $\r_1^3$ is a  non-degenerate rotational surface  with zero mean curvature everywhere.
\end{definition}

We notice that a transformation of a catenoid by a homothety of $\r_1^3$ gives other catenoid invariant by  the same group of rotations and with the same causal character. A straightforward computation leads to all catenoids in $\r_1^3$, obtaining the next classification (see \cite{ko,lo}).

\begin{theorem} \label{t-rot} Up to an isometry of $\r_1^3$ and assuming that the  rotational surface is parametrized  according Proposition \ref{p-revolution}, a catenoid in $\r_1^3$ is generated by one of the following profile curves (see Table \ref{table2}):
\begin{enumerate}
\item Elliptic catenoid. Then $f(s)=(1/a)\sinh(as+b)$ (spacelike surface) or $f(s)=(1/a)\sin( as+b)$ (timelike surface), $a\not=0$, $b\in\r$.
\item Hyperbolic catenoid.
\begin{enumerate}
\item Type I. Then  $f(s)=(1/a)\cosh(as+b)$ (timelike surface), $a\not=0$, $b\in\r$. There are not spacelike surfaces.
\item Type II. Then $f(s)=(1/a)\sin( as+b)$ (spacelike surface) and  $f(s)=(1/a)\sinh( as+b)$ (timelike surface), $a\not=0$, $b\in\r$.

\end{enumerate}
\item Parabolic catenoid. Then $f(s)=as^3+b$ (spacelike surface) and $f(s)=-as^3+b$ (timelike surface), where $a>0$, $b\in\r$.
\end{enumerate}
\end{theorem}

\begin{table}[htbp]
\centering \renewcommand{\arraystretch}{2}
\begin{tabular}{|l|c|c|c|c|}
\hline & \multicolumn{4}{c|}{Types of rotational surfaces} \\
\cline{2-5}
& &\multicolumn{2}{|c|}{hyperbolic}&\\
\cline{3-4}
     & elliptic & type I  & type II& parabolic \\\hline
    spacelike  & $\dfrac{1}{a}\sinh( as+b)$	&	 $-$&	 $\dfrac{1}{a}\sin( as+b)$ & $as^3+b$\\       \hline
  timelike & $\dfrac{1}{a}\sin( as+b)$	&	 $ \dfrac{1}{a}\cosh(as+b)$ &	 $\dfrac{1}{a}\sinh( as+b)$& $-as^3+b$\\
    \hline
\end{tabular}
\caption{Profile curves of the catenoids in $\r_1^3$}\label{table2}
\end{table}


\section{Elliptic catenoids spanning two coaxial circles}\label{s-elliptic}

We consider two coaxial circles $C_1\cup C_2$ with respect to the axis $L=\mbox{sp}\{e_3\}$. The analysis of how many catenoids connect $C_1$ with $C_2$ distinguishes  two cases according to the causal character of the surface.

\subsection{Spacelike surfaces}

After a translation and a homothety, we suppose that $C_1$ is the circle of radius $1$ in the $xy$-plane. Let $(1,0,0)$ and  $(x_0,0,z_0)$ be the intersection points between $C_1$ and $C_2$ with the $xz$-plane, respectively. Because the profile curve of the spacelike elliptic catenoid is given in terms of the $\sinh$ function, and $(1,0, 0)$ belongs to the surface, then the profile curve is
$$x=\frac{1}{a}\sinh (\pm az+\sinh^{-1}(a)).$$
 After a change of coordinates, the problem is formulated in terms of planar curves in the $xy$-plane as follows: given a point $P=(x_0,y_0)$, among the  curves in the family
${\mathcal F}=\{(1/a)\sinh(\pm ax+\sinh^{-1}(a)):a>0\}$ passing through the point $Q=(0,1)$, how many of such curves  does the point $P$ contain? Here $x_0\not=0$ (to be distinct of $Q$) and $y_0\not=0$ because $P$ does not belong to the rotation axis, namely, the $x$-axis.

Define the region $R=R_1\cup R_2\subset\r^2$ given by
$$R_1=\{(x,y)\in\r^2:x>0,y>x+1\},\ \ R_2=\{(x,y)\in\r^2:x<0,y<x+1\}$$
and let $T$ be  the symmetry of $R$  with respect to $y$-axis. Then all the curves   of ${\mathcal F}$ are contained in the region   $S=R\cup T\cup\{Q\}$. On the other hand, by the symmetry of the problem and the graphics of the elements of ${\mathcal F}$, the point $P$ must belong to $S\cup \Phi(S)$, where $\Phi$ is the symmetry of the $xy$-plane with respect to the $x$-axis. As a conclusion, the point $P$ can not belong to $\r^2-(S\cup\Phi(S))$, proving that   there is not a catenoid connecting $C_1$ and $C_2$. This gives a part of the statement of Theorem \ref{t-main}.

Suppose now that $P\in R\cup T$.  Since the circle generated by the point $(-x_0,y_0)$ is the same one than $P$, it is sufficient to consider the case of $P\in R$ and we are looking for an element of ${\mathcal F}$ of type $(1/a)\sinh(az+b)$ with $a>0$. Under this assumption on the point $P$, we will prove that there is exactly one curve of ${\mathcal F}$ going through $P$.

\textbf{Subcase $P\in R_1$}. Then there exists a curve of  ${\mathcal F}$ passing $P$ if there exists   $a\in\r$ that is a solution of the equation
$$\frac{1}{a}\sinh(ax_0+\sinh^{-1}(a))=y_0,$$
where $x_0>0$ and $y_0>1+x_0$, or equivalently,
\begin{equation}\label{se1}
\cosh(a x_0)+\frac{\sqrt{1+a^2}}{a}\sinh(a x_0)=y_0.
\end{equation}
 Define
\begin{equation}\label{gg}
g(a)=\cosh(a x_0)+\frac{\sqrt{1+a^2}}{a}\sinh(a x_0).
\end{equation}
As $\lim_{a\rightarrow 0}g(a)=1+x_0$ and $\lim_{a\rightarrow \infty}g(a)=+\infty$, we conclude by continuity that  there exists $a>0$ such that
$g(a)=y_0$. This proves that at least there is an element of ${\mathcal F}$ connecting both points.

To prove that there is exactly one, we see that the function $g$ is strictly increasing on $a$. The derivative of $g$  is
\begin{equation}\label{el1}
g'(a)=\frac{1}{a^2\sqrt{1+a^2}}\left(a(1+a^2) x_0\cosh(ax_0)+(a^2 x_0 \sqrt{1+a^2}-1)\sinh(ax_0)\right).
\end{equation}
We now show that the expression inside the brackets is positive for all $a,x\in(0,\infty)$. Define
\begin{equation}\label{hhh}
h(x)=a(1+a^2) x\cosh(ax)+(a^2 x\sqrt{1+a^2}-1)\sinh(ax).
\end{equation}
Then $h(0)=0$ and
\begin{equation}\label{el2}
h'(x)=a^2(1+x\sqrt{1+a^2})\left(a\cosh(ax)+\sqrt{1+a^2}\sinh(ax)\right).
\end{equation}
As $h'(x)>0$ for $a,x>0$, then $h$ is strictly increasing on $x$, so $h(x)>h(0)=0$, proving that \eqref{el1} is positive.

\textbf{Subcase $P\in R_2$}. We prove the existence of a value $a\in\r$ that is a solution of  \eqref{se1} for $x_0<0$ and $y_0<x_0+1$. With the same function $g$ defined in \eqref{gg}, we have  $\lim_{a\rightarrow 0}g(a)=1+x_0$ and $\lim_{a\rightarrow \infty}g(a)=-\infty$ and this shows the existence of a solution $a$ of  \eqref{se1}, so there is a curve in the family ${\mathcal F}$ passing through the point $P$. Proving the uniqueness of this catenoid is equivalent to see that  $g$ is strictly decreasing. For this, we  show that $g'(a)<0$, or equivalently, that $h(x)<0$ for $a,-x\in (0,\infty)$, where $h$ is defined in \eqref{hhh}.  A simple study of the function $h(x)$ proves that $h(0)=0$, $h<0$ in a neighborhood of $(-\epsilon,0)$ of $x=0$  and  when $x<0$, the function $h$ has exactly a local maximum $x_M$ and a local minimum $x_m$, where $x_M$ and $x_m$ are determined by
$$\tanh(ax_M)=-\frac{a}{\sqrt{1+a^2}},\ \ x_m=-\frac{1}{\sqrt{1+a^2}}.$$
The value $h$ at $x_M$ is
$$h(x_M)=a\left( x_M+\frac{1}{\sqrt{1+a^2}}\right)\cosh(ax_M)<0,$$
because $x_M<x_m=-1/\sqrt{1+a^2}$. This proves that $h(x)<0$ for $x\in (-\infty,0)$.

As conclusion, we have proved:
\begin{theorem}\label{t-elliptic}
 Let $C_1\cup C_2$ be two Euclidean coaxial circles in $\r_1^3$ with respect to the $z$-axis. Then the number of spacelike elliptic catenoids spanning $C_1\cup C_2$ is $0$ or $1$.
\end{theorem}

\begin{remark}  A zero mean curvature spacelike surface is called a maximal surface because it maximizes the area. In fact, by  the expression of the  Jacobi operator associated to the second variation of the surface area, it is straightforward to check that the surface is stable  in a strong sense, that is, the first eigenvalue of the Jacobi operator
on any compact domain is positive \cite{bo}.  For this reason, and in contrast to the Euclidean setting described in the introduction, it was expected that  there would  be a unique catenoid at most connecting two coaxial circles.
\end{remark}

We analyse  the case of two coaxial circles with the same radius. Although this case is covered in Theorem \ref{t-elliptic},   in this particular case  we find a relation between the radius and the separation between the circles.

\begin{corollary}\label{c1} Let $C_{\pm h} =\{(x,y,z)\in\r_1^3: x^2+y^2=r^2, z= \pm  h\}$ be two circles of radius $r>0$ with respect to the $z$-axis and separated a distance $2h>0$ far apart. Then the number of spacelike elliptic catenoids connecting the circles $C_{-h}\cup C_h$ is $0$ if $r\leq h$ and it is $1$ if $r>h$.
\end{corollary}

\begin{proof}
We formulate the problem for planar curves    in the $xy$-plane. Let us fix the point $P=(h,r)$ and study how many   curves of type $f(x)=(1/a)\sinh(ax)$ go through the point $P$. This is equivalent to study the number of solutions of equation
\begin{equation}\label{ee1}
\frac{1}{a}\sinh(ah)=r,
\end{equation}
where $r$ is a given number and the unknown is $a$. Define the function $g(a)= \sinh(ah)/a$. We have $\lim_{a\rightarrow 0}g(a)=h$ and $\lim_{a\rightarrow\infty}g(a)=\infty$. Since
$$g'(a)=\frac{ah\cosh(ah)-\sinh(ah)}{a^2},$$
and the numerator is always positive, the function $g$ is strictly increasing, proving that there is a unique value $a$ reaching $g(a)=r$ only for $r>h$.

\end{proof}

\subsection{Timelike surfaces}

We consider two coaxial circles and we ask how many timelike elliptic catenoids connect both circles. In this subsection we focus in the case that the circles have the same radius which we suppose that, after a homothety,  $r=1$.  Up to a translation in the direction of the rotation axis, we also suppose that the circles are $C_{-h}\cup C_h$, where $C_{h}=\{(x,y, h)\in\r_1^3:x^2+y^2=1\}$. We know that the profile curve of a timelike elliptic catenoid is   $x=f(z)=(1/a)\sin (\pm az+b)$, where $a>0$ and $b\in \mathbb{R}$.   By the symmetry with respect to the axis of rotation,  the formulation  of the problem for planar curves is as follows: given the points $P=(-h,\pm 1)$ and $Q=(h,1)$, how many curves in the family ${\mathcal F}=\{(1/a)\sin (\pm ax+b): a,b\in\r, a>0\}$  connect both points. For  such a curve, the boundary conditions imply
$$\sin(\pm ah+b)=a,\ \ \sin (\mp ah+b)=\pm a.$$
\begin{enumerate}
\item Case $\sin (-ah+b)= a$. We obtain $b=(2k+1)\pi/2 $, $k\in \z$, or $ah=k\pi$, $k\in\n$.
\begin{enumerate}
\item Subcase $b=(2k+1)\pi/2 $. Then the curve is  $y(x)=(-1)^k\cos(ax)/a$. Independently of the value $(-1)^k$, $y(x)$ describes the same rotational surface, we suppose $k=0$, so   $y(x)=\cos(ax)/a$. Thus we ask on the number of values $a$ that are solutions of
$$\cos(ah)-a=0\ \ (a>0)$$
depending on the distance $2h>0$ between the circles  $C_{-h}$ and $C_h$. Since
$$\lim_{a\rightarrow 0}(\cos(ah)-a)=1,\ \ \lim_{a\rightarrow \infty}(\cos(ah)-a)=-\infty,$$
 we deduce  that there is at least  one solution. We now study the number of solutions depending on the value of $h$. The derivative of the function $g_h(a)=\cos(ah)-a$ is $g_h'(a)=-h\sin(ah)-1$. If $h\leq 1$, $g_h$ is decreasing with respect to $a$, so there is a unique curve of ${\mathcal F}$  connecting the points $P$ and $Q$. We study the zeroes of $g_h$.  First, we observe the next periodicity property on $g_h$:
\begin{equation}\label{pe}
g_h(a+\frac{2k\pi}{h})=g_h(a)-\frac{2k\pi}{h}\ \ (k\in \n).
\end{equation}
 We proved that if $h\leq 1$, the function $g_h$ is decreasing and there is a unique zero of $g_h$. After  $h>1$, the function $g_h$ has an infinite number of local minimum and maximum. By \eqref{pe}, and because the first local minimum (after the first zero of $g_h$) is negative, the rest of local minimum are negative. However, we will see that for $h$ sufficiently big, there is a finite number of local maximum where the value of $g_h$ is  positive, getting zeroes of $g_h$ between a local minimum and one of these local maximum by the Bolzano theorem. See Fig. \ref{ftime}. Let $m_0>0$ be the first minimum of $g_h$. Then we have $\sin(m_0h)=-1/h$. We point out that $m_0=m_0(h)$ decreases as $h$ increases, that is, $\lim_{h\rightarrow \infty}m_0(h)=0$.  We know that $\cos(m_0h)<0$. As a consequence of \eqref{pe}, the set of local minimum $\{m_k:k\in\n\}$ and the set of local maximum points $\{M_k:k\in\n\}$ of $g_h$ are given by the relations
 $$m_k=m_0+\frac{2k}{h}\pi,\ \ M_k=-m_0+\frac{2k+1}{h}\pi.$$
 Using that $\sin(m_0 h)=-1/h$, for $h>1$ we have  that
 \begin{equation}\label{fhm}
 g_h(M_k)=-\cos(m_0 h)+m_0-\frac{2k+1}{h}\pi=\frac{\sqrt{h^2-1}}{h}+m_0-\frac{2k+1}{h}\pi.
 \end{equation}

 \begin{figure}[hbtp]
\begin{center}
\hspace{0.5cm}\includegraphics[width=.3\textwidth]{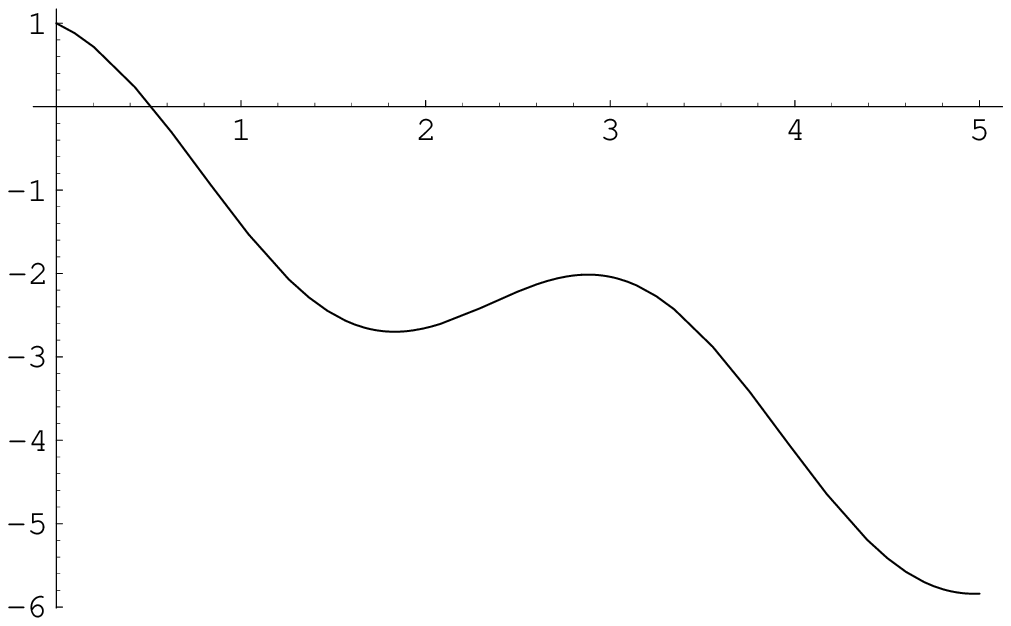}
\hspace{0.4cm}\includegraphics[width=.3\textwidth]{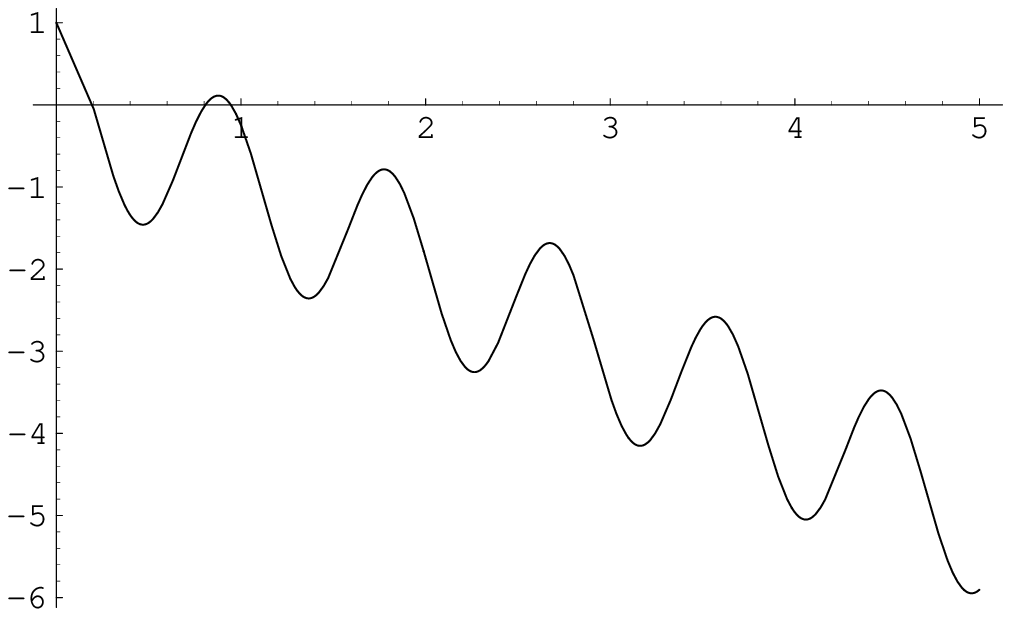}\hspace{0.4cm}\includegraphics[width=.3\textwidth]{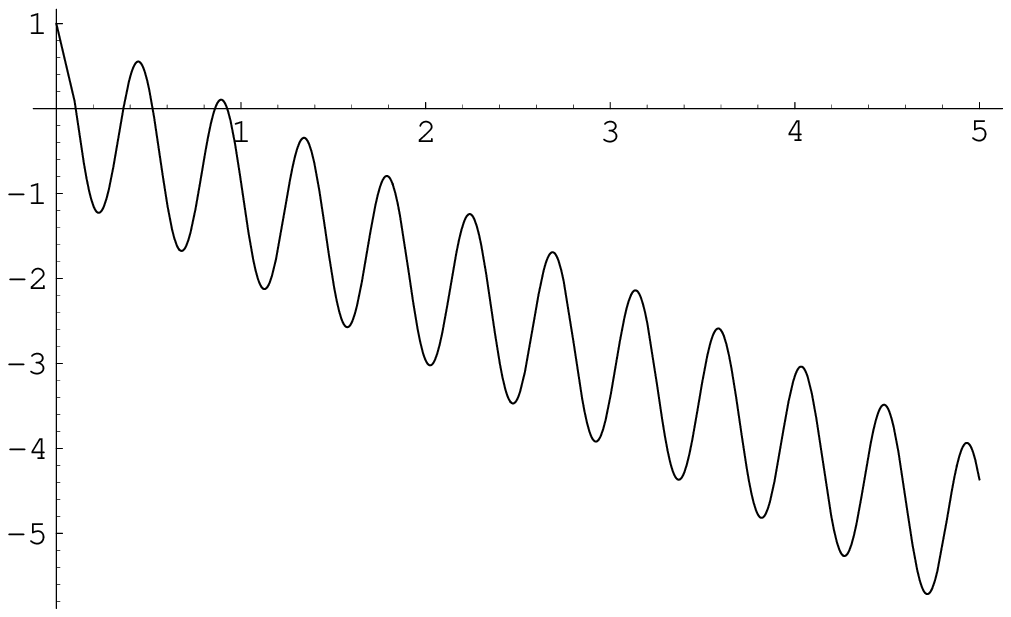}\end{center}
\caption{Graphics of $g_h$ for different  values of $h$: from left to right, $h=2,7$ and $14$}\label{ftime}
\end{figure}

Thus, for $h>1$ sufficiently big, $g_h(M_k)>0$ and Eq. \eqref{fhm} implies that there exists $k_0\geq 1$ such that
$g_h(M_k(h))>0$ for $1\leq k\leq k_0$ and for $k>k_0$, the value $g_h(M_k)$ is negative by \eqref{pe}.
A numerical computation gives the first value $h$ where there are exactly two zeroes of $g_h$, namely,  $h\simeq 6.202$ and $g_h(M_1(h))=0$.

\item Subcase $ah=k\pi$, $k\in\n$. Then
$$y(x)=\frac{h}{k\pi}\sin(\pm\frac{k\pi}{h}x+b).$$
Because $y(h)=1$, then
\begin{equation}\label{sinb}
\sin(b)=(-1)^k\frac{k\pi}{h}.
\end{equation}
It is clear that if $h<\pi$, there   exists no a solution of \eqref{sinb}. If $h\geq \pi$, we have to solve $\sin(b)=(-1)^kk\pi/h$. For $h\geq k\pi$, there exists many $b$'s solving  $\sin(b)=(-1)^kk\pi/h$. We observe that the first value where there is at least one solution is $h=\pi$ and the number of solutions increases as $h$ increases.
\end{enumerate}

\item Case $\sin(\mp ah+b)=-a$. Then $b=k\pi$, $k\in\z$ or $ah=k\pi+\pi/2$, $k\in\n\cup\{0\}$.
\begin{enumerate}
\item Subcase $b=k\pi$, then $$y(x)=\frac{1}{a}\sin (\pm ax+k\pi)=\frac{(-1)^{k }}{a}\sin(\pm ax).$$
We solve $\sin(ah)-a=0$ where the unknown is $a$. Define the function $G_h(a)=\sin(ah)-a$, which satisfies
$G_h(0)=0$. If $h\leq 1$, then $G_h'(a)=h\cos(ah)-1\leq 0$ and $G_h$ is decreasing and this proves that there is not a solution. If $h>1$, then $G_h$ is increasing in an interval close to $a=0$, so $G_h$ is positive in this interval. Since $\lim_{a\rightarrow\infty}G_h(a)=-\infty$, then there is a solution, proving that there is at least a catenoid.

The behavior of the function $G_h$ is similar to $g_h$ because it holds the relation
$$G_h(a+\frac{\pi}{2h})=g_h(a)-\frac{\pi}{2h}.$$
When $h>1$,   $G_h$ has an infinite number of local minimum and maximum. The first critical point corresponds with a local maximum at $M_0^*$ with $G_h(M_0^*)>0$. We study when the function $G_h$ at the second local maximum $M_1^*$ is positive and this occurs when $h\simeq 7.790$,  so $G_h(M_1(h)^*)=0$, obtaining two catenoids.
The conclusions are the same as in the first subcase previously studied
\item Subcase $ah=k\pi+\pi/2$, $k\in\n\cup\{0\}$. Then $$y(x)=\frac{2h}{(2k+1)\pi}\sin\left(\pm \frac{(2k+1)\pi}{2h}x+b\right).$$
The condition $y(h)=1$ gives
$$\cos(b)=\pm (-1)^{k}\frac{(2k+1)\pi}{2h}.$$
If $h<\pi/2$, there is not a solution and for $h\geq \pi/2$, there is at least one solution. The number of solutions increases as $h$ increases.\end{enumerate}
\end{enumerate}

We summarize the above results.

\begin{theorem}\label{t-elliptic2} Let $C_{-h}$ and $C_{h}$ be two coaxial Euclidean circles of radius $r>0$ with respect to the $z$-axis and separated a distance equal to $2h$. If $N(h)$ denotes  the number of timelike elliptic catenoids connecting $C_{-h}$ and $C_h$, then: 
\begin{enumerate}
\item $N(h)$ is a finite number.
\item  $N(h)$ is a non-decreasing function  on $h$.
\item The limit of $N(h)$ is $\infty$ as $h\rightarrow\infty$.
\item  There exists $c_0>0$ such that if  $h/r<c_0$, then $N(h)=1$.
\end{enumerate}
\end{theorem}

By the above proof, we have the next information about the number of catenoids connecting$C_{-h}$ and $C_h$:
\begin{enumerate}
\item[Case (1)]
\begin{enumerate}
\item  If $h/r <6.202$, then there is not a catenoid.
\item    If  $h/r<\pi$, then there is not a catenoid.
\end{enumerate}
\item[Case (2)]
\begin{enumerate}
\item If $h/r\leq 1$, then the number of catenoids is $0$,  and if $1<h/r<7.790$, then it is $1$ .
\item If  $h/r<\pi /2$, then there is no a catenoid.
\end{enumerate}
\end{enumerate}
Using the inequality $1<\pi /2<\pi<6.202<7.790$, we obtain that if $h/r>1$, the number of catenoids is more than 1 and we can take at least two catenoids which correspond to the cases (1,a) and (2,a).

\begin{remark} We see that the behavior of the number of timelike elliptic catenoids connecting two coaxial circles with the same radius is the opposite to the Euclidean case: as we increase the separation between the circles, the number of  catenoids connecting them increases. See Fig. \ref{tec}.
\end{remark}

\begin{figure}[hbtp]
\begin{center}
\includegraphics[width=.21\textwidth]{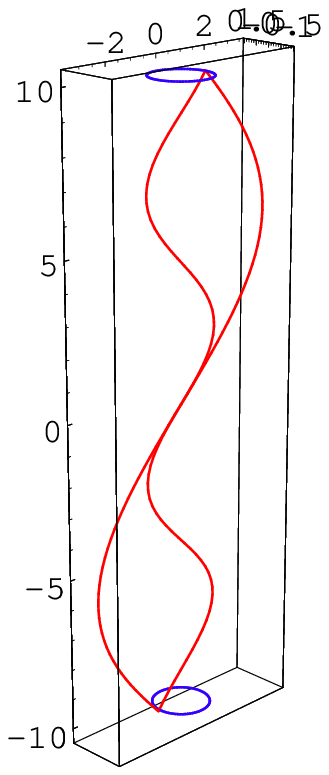}\hspace*{.2\textwidth} \includegraphics[width=.27\textwidth]{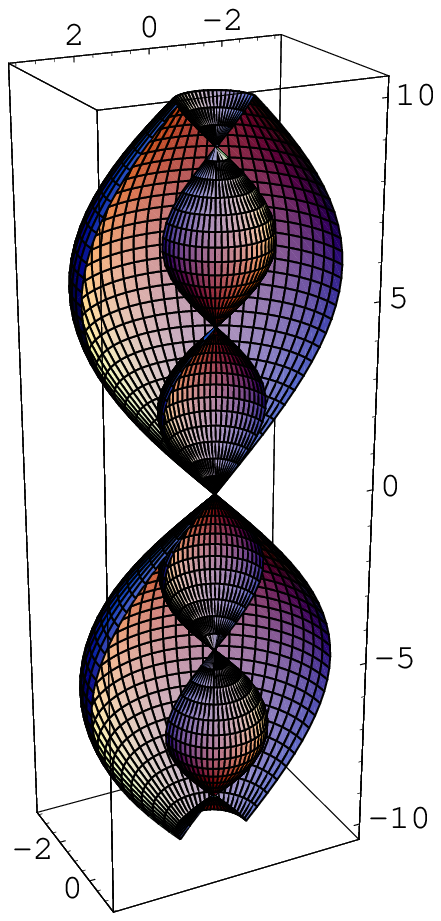}\end{center}
\caption{Two timelike elliptic catenoids connecting two coaxial circles. Left: two coaxial circles of radius $r=1$ and $h=20$ far apart (blue) connected by two profile curves $f(s)=\sin(as)/a$ for values $a\simeq 0.285$ and $a\simeq 0.706$ (red). Right: the corresponding two timelike elliptic catenoids}\label{tec}
\end{figure}
\section{Hyperbolic catenoids  spanning two coaxial circles}

In this section, we consider two coaxial hyperbola $C_1\cup C_2$ and we ask how many of hyperbolic catenoids connect $C_1$ with $C_2$. Using the same terminology of Proposition \ref{p-revolution}, we say that $C_1\cup C_2$ are two coaxial hyperbolas of type I (resp. type II) if there exists a hyperbolic rotational surface of type I (resp. type II) connecting $C_1$ with $C_2$. In particular, $C_1$ and $C_2$ are timelike hyperbolas (resp. spacelike hyperbolas). The profile curves are given in Theorem \ref{t-rot}. Exactly, the profile curve of the hyperbolic catenoid of type I is $y(x)=(1/a)\cosh(ax+b)$, $a\not=0$, $b\in\r$, which coincides with the catenary in the Euclidean setting and whose behavior has been described in the introduction for coaxial circles with the same radius. When the surface is of type II, the profile curves have appeared in the elliptic case. Thus we have:

\begin{theorem}\label{t-hyperbolic}   Let $C_1\cup C_2$ be two coaxial timelike hyperbolas in $\r_1^3$ with respect to the $x$-axis. Then the number of timelike hyperbolic catenoids of type I spanning $C_1\cup C_2$ is $0$, $1$ or $2$.
\end{theorem}

The statement of this theorem needs to be clarified in the following sense. In Euclidean space, and for $a>0$, $b\in\r$, the two catenoids obtained by rotating about the $z$-axis the curve $\{x(z)=(1/a)\cosh(az+b), y=0\}$ and  the profile curve $\{x(z)=(-1/a)\cosh(az+b), y=0\}$ coincide because the circles of the catenoid are Euclidean circles. However in $\r_1^3$, for timelike hyperbolic catenoids of type I, the corresponding catenoids are different. Exactly, the catenoids
$$S_a=\{(s,\frac{1}{a}\cosh(as+b)\cosh(t),\frac{1}{a}\cosh(as+b)\sinh(t): s,t\in\r\}$$
$$S_{-a}=\{(s,-\frac{1}{a}\cosh(as+b)\cosh(t),-\frac{1}{a}\cosh(as+b)\sinh(t): s,t\in\r\}$$
are separated by the plane $\Pi_1$ of equation $y=0$, with $S_a\subset\{y>0\}$ and $S_{-a}\subset \{y<0\}$. Therefore, in Theorem \ref{t-hyperbolic} we are assuming that the coaxial circles $C_1\cup C_2$ lie in the same side of the plane $\Pi_1$. For example, taking $a=1$, $b=0$, the timelike hyperbolas of type I given by 
$C_1=\{(h,\cosh{t},\sinh{t})\in\r_1^3: t\in \mathbb{R}\}$ and $C_2=\{(-h,-\cosh{t},\sinh{t})\in\r_1^3: t\in \mathbb{R}\}$, $h>0$, are separated a distance $2h>0$. Although  they are  invariant by the group of rotations whose axis is $L=\mbox{sp}\{e_1\}$,   they can not be connected by a hyperbolic catenoid of type I {\it for every value $h$}.  Therefore if we want to state a similar result as in Euclidean space relating the distance between the hyperbolas and the existence of a catenoid connecting them, we have to add the assumption that they lie in the same side of $\Pi_1$. Thus we have:

\begin{corollary}
 Let $C_{\pm h}=\{(\pm h,r\cosh{t},r\sinh{t})\in\r_1^3: t\in \mathbb{R}\}$ be two coaxial timelike hyperbolas of radius $r>0$ and separated $2h>0$ far apart.  Then there is a number $c_1\simeq 1.325$ such that the number of timelike hyperbolic catenoids of type I connecting $C_{-h}\cup C_h$ is $0$, $1$ or $2$ depending if  $h/r>c_1$, $h/r=c_1$ and  $ h/r<c_1$, respectively.
\end{corollary}

\begin{figure}[hbtp]
\begin{center}\hspace{1.4cm} \includegraphics[width=.4\textwidth]{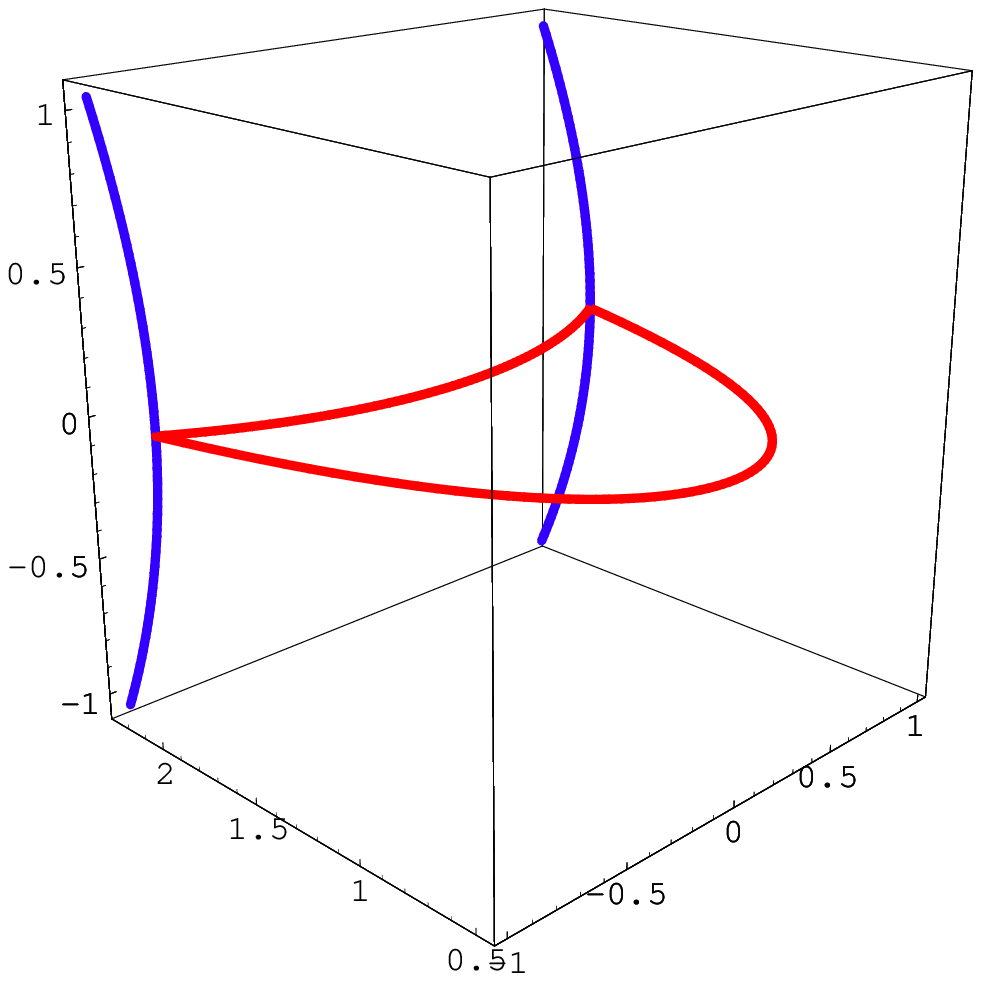}\hspace{1.cm} \includegraphics[width=.4\textwidth]{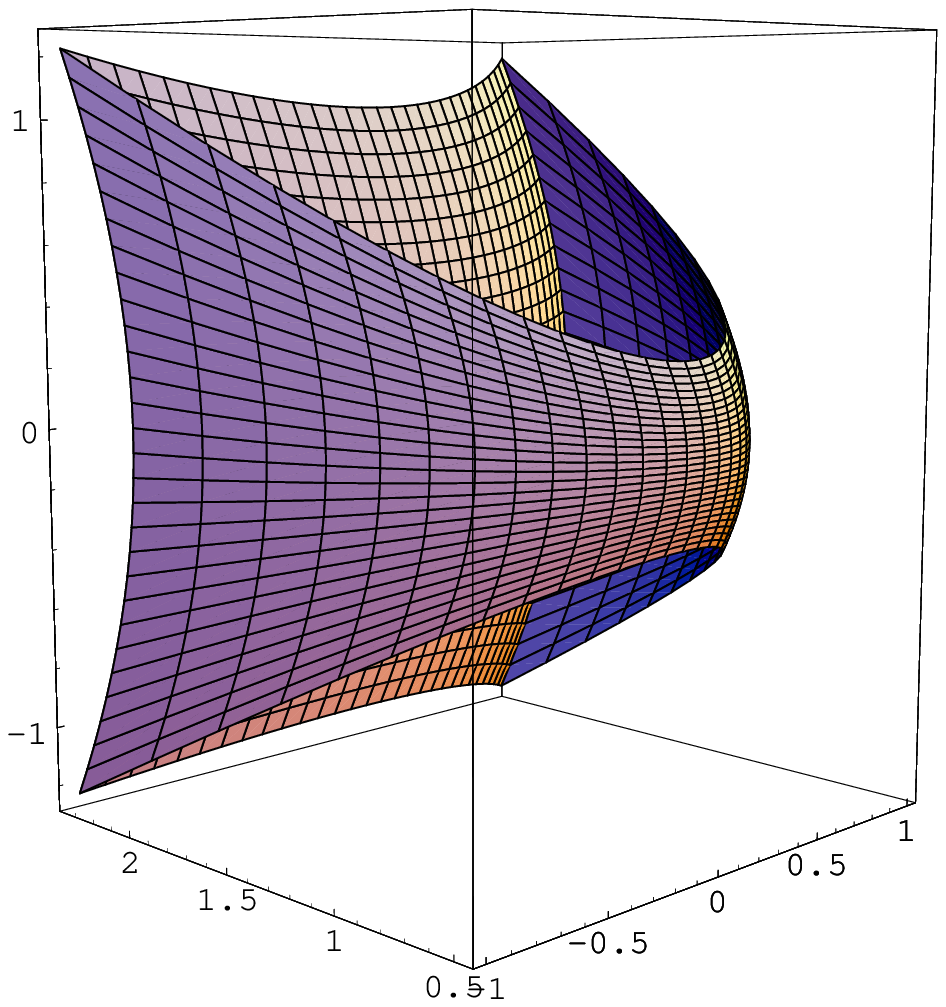}\end{center}
\caption{Two hyperbolic catenoids of type I connecting two coaxial circles. Left: two coaxial hyperbolas of type I of radius $r=2$ and $h=2$ far apart (blue) connecting by two profile curves $f(s)=\cosh(as)/a$ for values $a\simeq 0.589$ and $a\simeq 2.126$ (red). Right: the corresponding two hyperbolic catenoids of type I}\label{thc}
\end{figure}

We now consider two coaxial spacelike hyperbolas of type II. Similarly to the case of Theorem \ref{t-hyperbolic}, there exist spacelike hyperbolas of type II that can not be connected by a timelike hyperbolic catenoid of type II. For example, this occurs with the hyperbolas $C_1=\{(h,\sinh{t},\cosh{t})\in\r_1^3: t\in \mathbb{R}\}$ and $C_2=\{(-h,\sinh{t},\cosh{t})\in\r_1^3: t\in \mathbb{R}\}$ which are in the same side of the plane $\Pi_2$ of  equation $z=0$. For the case of spacelike hyperbolic catenoids of type II, this phenomenon does not occur because the profile curve is given in terms of the sine function. Therefore, in (1) of Theorem \ref{t-hyperbolic2} below, we are assuming that the circles lie in opposite sides of $\Pi_2$. 

\begin{theorem}\label{t-hyperbolic2} Let $C_1\cup C_2$ be two coaxial spacelike hyperbolas of type II. Then:
\begin{enumerate}
\item  The number of timelike hyperbolic catenoids of type II spanning $C_1\cup C_2$ is $0$ or $1$.
 \item If the radius of $C_1$ and $C_2$ coincide, and $h$ is the distance separating $C_1$ and $C_2$, then there exist at least one spacelike hyperbolic catenoid of type II spanning $C_1\cup C_2$ and the number $N(h)$ of these catenoids connecting $C_1$ and $C_2$ increases (going to $\infty$) as $h\rightarrow\infty$.
        \end{enumerate}
\end{theorem}

As a consequence of the argument in Corollary \ref{c1}, we obtain:

\begin{corollary}
  Let $C^{\pm}_{h}=\{(\pm h,r\sinh{t},\pm r\cosh{t})\in\r_1^3: t\in \mathbb{R}\}$   be two coaxial spacelike hyperbolas of radius $r>0$ and separated $2h>0$ far apart.  Then the number of timelike hyperbolic catenoids of type II connecting $C^{+}_{h}$ and $C^{-}_{h}$  is $0$ if $r\leq h$ and it is $1$ if $r>h$.
\end{corollary}

\section{Parabolic catenoids spanning two coaxial circles}

Consider a parabolic catenoid in $\r_1^3$. By Theorem \ref{t-rot}, we know that the profile curve of a spacelike surface is  $f(s)=as^3+b$ ($a>0$), and if it is timelike, then  $f(s)=-as^3+b$ ($a>0$). The surface has a singularity when it meets with the rotation axis, that is, when $s=0$.

\begin{theorem}\label{t-parabolic}
Given two coaxial  parabola circles, there is $0$ or $1$ (spacelike or timelike)  parabolic catenoid connecting both circles.
\end{theorem}

\begin{proof}
\begin{enumerate}
\item Spacelike case. The generating curve is $f(s)=as^3+b$ with $a>0$. The formulation of the problem for planar curves is as follows. Given two points $P$ and $Q$ of the $xy$-plane, find how many curves of type $f(x)=ax^3+b$, $a>0$, pass through $P$ and $Q$. The rotation axis of the surface corresponds with the $y$-axis.  After a vertical translation and a homothethy, we suppose that  $P=(1,0)$. Then $f(x)=a(x^3-1)$. Let $Q=(x_0,y_0)$ be another point of the $xy$-plane. The problem reduces to find how many curves of the family ${\mathcal F}=\{f(x)=a(x^3-1):a>0\}$ go through the point $Q$. By the graphics of the elements of ${\mathcal F}$,  a first necessary condition is that  $Q$ must belong to the region $R=R_1\cup R_2$, where
    $$R_1=\{(x,y)\in\r^2: x-1>0,y>0\},\ \  R_2=\{(x,y)\in\r^2: x-1<0,y<0\}.$$
In particular, if $Q\not\in R$, there is not a spacelike parabolic catenoid connecting both circles. This proves the part of Theorem \ref{t-parabolic} that asserts that the number of catenoids connecting two circles is $0$. Suppose now that  $Q\in R_1$. We will find a value $a$ such that $a(x_0^3-1)=y_0$ and next, study how many values $a$ satisfy this equation. The function $g(a)=a(x_0^3-1)$ satisfies $\lim_{a\rightarrow 0}g(a)=0$ and $\lim_{a\rightarrow \infty}g(a)=\infty$. Since $0<y_0$, a continuity argument proves that there exists a value $a$ such that $g(a)=y_0$. On the other hand, the derivative $g'(a)=x_0^3-1$ is positive, proving that $g$ is strictly increasing, so the uniqueness of the curve among all ones of the family ${\mathcal F}$.  It follows the existence of a unique catenoid connecting the corresponding two circles. The argument when $Q\in R_2$ is similar. This finishes the proof of Theorem \ref{t-parabolic} for a spacelike surface.

\item Timelike case. The generating curve is $f(s)=as^3+b$ with $a<0$. Again, we suppose that $P=(1,0)$ so   $f(x)=a(x^3-1)$. Let $Q=(x_0,y_0)$ be another point of the $xy$-plane. The same argument as above proves that if $Q$ belongs to the region   $T=T_1\cup T_2$, where
    $$T_1=\{(x,y)\in\r^2: x-1>0,y<0\},\ \  T_2=\{(x,y)\in\r^2: x-1<0,y>0\},$$
there exists a unique curve going through $Q$ and if $Q\not\in T$, then there is not a such curve.
\end{enumerate}
\end{proof}

\textbf{Acknowledgement.} The first author has been supported by the Research Fellowships of the Japan Society for the Promotion of Science (JSPS). The second author has been partially supported by  the MINECO/FEDER grant MTM2014-52368-P.

\bibliographystyle{elsarticle-num}

 \end{document}